\documentclass[twoside]{article}
\usepackage{amsmath,amsfonts,amssymb,amsthm}
\usepackage{hyperref}
\usepackage[a4paper]{geometry}
\usepackage[color,matrix,arrow,all]{xy}
\usepackage{stackrel,relsize}
\newcommand{\ds}{\displaystyle\sum}

\newcommand{\rat}{\mathbb{Q}}

\newcommand{\complex}{\mathbb{C}}
\newcommand{\nat}{\mathbb{N}}
\newcommand{\integ}{\mathbb{Z}}

\newcommand{\op}{\operatorname}

\newcommand{\Hom}{\operatorname{Hom}}

\newcommand{\Ext}{\operatorname{Ext}}
\newcommand{\GL}{\operatorname{GL}}
\newcommand{\SL}{\operatorname{SL}}
\newcommand{\Rep}{\operatorname{Rep}}
\newcommand{\SI}{\operatorname{SI}}

\newtheorem{theorem}{Theorem}[section]

\newtheorem{lemma}[theorem]{Lemma}
\newtheorem{proposition}[theorem]{Proposition}
\newtheorem{corollary}[theorem]{Corollary}

\newtheorem{conjecture}{Conjecture}
\theoremstyle{definition}
\newtheorem{remark}[theorem]{Remark}
\newtheorem{definition}[theorem]{Definition}
\newtheorem{notation}[theorem]{Notation}

\newtheorem{example}[theorem]{Example}

\def\presuper#1#2%
  {\mathop{}%
   \mathopen{\vphantom{#2}}^{#1}%
   \kern-\scriptspace%
   #2}

\begin{document}

\title{Singularities of zero sets of semi-invariants for quivers}
\author{Andr\'as Cristian L\H{o}rincz}
\date{}

\maketitle

\begin{abstract}
Let $Q$ be a quiver with dimension vector $\alpha$ prehomogeneous under the action of the product of general linear groups $\GL(\alpha)$ on the representation variety $\Rep(Q,\alpha)$. We study geometric properties of zero sets of semi-invariants of this space. It is known that for large numbers $N$, the nullcone in $\Rep(Q,N\cdot \alpha)$ becomes a complete intersection. First, we show that it also becomes reduced. Then, using Bernstein-Sato polynomials, we discuss some criteria for zero sets to have rational singularities. In particular, we show that for Dynkin quivers codimension $1$ orbit closures have rational singularities.
\end{abstract}


\vspace*{5mm}

\section*{Introduction}
\label{sec:intro}

The study of zero sets of semi-invariants for quivers has been initiated in \cite{chang}, and has been intensively investigated later in several articles. In particular \cite{zwara1} shows that the nullcone for prehomogeneous dimension vectors is an irreducible complete intersection if the dimension vector is not \lq\lq too small". Bounds have been given for tame quivers in \cite{zwara2}. In Section \ref{sec:redu} we state analogous results concerning whether the nullcone is reduced. We note that such questions have been investigated before outside the quiver setting (for example, see \cite{kraft}).  

In the first part of \ref{sec:redu} we translate Serre's criterion in the quiver setting, see Definition \ref{jacob}. The main results of this section are Theorem \ref{thm:reduc} and Theorem \ref{thm:dynk}. Clearly, results about the nullcone are valid for the zero set of an arbitrary set of fundamental semi-invariants.

In Section \ref{sec:bern} we use $b$-functions and Bernstein-Sato polynomials to give some results on whether zero sets have rational singularities (and in particular are normal). This is based on the calculation of $b$-functions (of several variables) from \cite{lol}. We note that zero sets often turn out to be orbit closures (see \cite{mater,mater2}) and the hypersurface case corresponds to codimension $1$ orbits if the dimension vector is not \lq\lq too small" (see \cite{openorb}). In the latter case the results are sharper and we prove Theorems \ref{thm:dynkin1} and \ref{thm:extend1} for tame quivers. For zero sets of more semi-invariants, we establish an elementary link between the $b$-function of several variables and the Bernstein-Sato polynomial of an ideal (Proposition \ref{prop:link}). Then we state Lemma \ref{lem:reduc} and exhibit how it can be used in Example \ref{ex:first}. Lastly, we state Theorem \ref{thm:tree} involving some special semi-invariants for tree quivers.

It is known that orbit closures of type $A,D$ quivers have rational singularities (see \cite{orb1,orb2}). The roots of $b$-functions are finer invariants that can be used to distinguish orbit closures. In particular, we see in Remark \ref{rem:sing} that orbit closures of type $A$ and $D$ are not smoothly equivalent.

\begin{notation}
As usual, $\nat$ will denote the set of all non-negative integers and $\rat_+$ (resp. $\rat_-$) the set of non-negative (resp. non-positive) rational numbers. 

For $a,b,d\in \integ, a\leq b$, we use the following notation in $\complex[s]$: 
$$[s]^d_{a,b}:=\prod_{i=a+1}^b \prod_{j=0}^{d-1} (ds+i+j).$$
If $a=0$, we sometimes write $[s]^d_{b}:=[s]^d_{0,b}$. Hence $[s]^{d}_{a,b}[s]^d_{a}=[s]^d_b$.

Now fix an $l$-tuple $\underline{m}=(m_1,\dots,m_r)\in \nat^r$. Then for any $r$-tuple $\underline{d}=(d_1,\dots,d_r)\in \nat^r$, we use the following notation in $\complex[s_1,\dots,s_r]$:
$$[s]^{\underline{d}}_{a,b}=\prod_{i=a+1}^b \prod_{j=0}^{d-1} (\underline{d}\cdot\underline{s}+i+j),$$
where $d=\underline{d}\cdot \underline{m}$, the usual dot product, and $\underline{s}=(s_1,\dots,s_r)$.
\end{notation}

\section{Preliminaries}
\label{sec:semi}

Let $k$ be an algebraically closed field. A quiver $Q$ is an oriented graph, i.e. a pair $Q=(Q_0,Q_1)$ formed by a finite set of vertices $Q_0=\{1,\dots,n\}$ and a finite set of arrows $Q_1$. An arrow $a$ has a head $ha$, and tail $ta$, that are elements in $Q_0$:

\[\xymatrix{
ta \ar[r]^{a} & ha
}\]

A representation $V$ of $Q$ is a family of finite dimensional vector spaces $\{V(x)\,|\, x\in Q_0\}$ together with linear maps $\{V(a) : V(ta)\to V(ha)\, | \, a\in Q_1\}$. The dimension vector $\underline{d} (V)\in \nat^{Q_0}$ of a representation $V$ is the tuple $\underline{d}(V):=(d_x)_{x\in Q_0}$, with $d_x=\dim V(x)$. A morphism $\phi:V\to W$ of two representations is a collection of linear maps $\phi = \{\phi(x) : V(x) \to W(x)\,| \,x\in Q_0\}$, with the property that for each $a\in Q_1$ we have $\phi(ha)V(a)=W(a)\phi(ta)$. Denote by $\Hom_Q(V,W)$ the vector space of morphisms of representations from $V$ to $W$. We say $Q$ is tame if it is of Dynkin or extended Dynkin type (for more on quivers cf. \cite{elements}).

We form the affine space of representations with dimension vector $\alpha\in \nat^{Q_0}$ by
$$\Rep(Q,\alpha):=\displaystyle\bigoplus_{a\in Q_1} \Hom(k^{\alpha_{ta}},k^{\alpha_{ha}}).$$
The group 
$$\GL(\alpha):= \prod_{x\in Q_0} \GL(\alpha_x)$$
acts on $\Rep(Q,\alpha)$ in the obvious way. Furthermore, we also consider the subgroup $\SL(\alpha):= \prod_{x\in Q_0} \SL(\alpha_x)$.

\vspace{0.1in}

Under the action $\GL(\alpha)$ two elements lie in the same orbit iff they are isomorphic as representations. 

For two vectors $\alpha, \beta\in \integ^{Q_0}$, we define the Euler product
$$\langle \alpha, \beta \rangle = \ds_{x\in Q_0} \alpha_x \beta_x - \ds_{a\in Q_1} \alpha_{ta} \beta_{ha}.$$
For any two representations $V$ and $W$, we have the following exact sequence:

\begin{equation}\label{eq:ringel}
\begin{array}{rlc}
0 \to \Hom_Q (V,W) \stackrel{i}{\longrightarrow} \displaystyle\bigoplus_{x \in Q_0}& \!\!\!\!\!\Hom(V(x),W(x)) & \\
& \stackrel{d^V_W}{\longrightarrow}  \displaystyle\bigoplus_{a\in Q_1} \Hom(V(ta),W(ha)) \stackrel{p}{\longrightarrow} \Ext_Q(V,W)\to 0 &
\end{array}
\end{equation}

Here, the map $i$ is the inclusion, $d_W^V$ is given by
$$\{\phi(x)\}_{x\in Q_0} \mapsto \{\phi(ha)V(a) - W(a)\phi(ta)\}_{a\in Q_1}$$
and the map $p$ builds an extension of $V$ and $W$ by adding the maps $V(ta)\to W(ha)$ to the direct sum $V\oplus W$.

The exact sequence (\ref{eq:ringel}) gives us $\langle \underline{d}(V),\underline{d}(W) \rangle = \dim \Hom(V,W) - \dim \Ext(V,W)$.

For a geometric interpretation of this sequence, take $V=W$. Then $\bigoplus_{x \in Q_0} \Hom(V(x),V(x))\cong \mathfrak{gl}(\underline{d}(V))$, the Lie algebra of $\GL(\underline{d}(V))$. The map $d^V_V$ is the differential at the identity of the orbit map
$$g\mapsto g\cdot V \in \Rep(Q,\underline{d}(V)).$$
Also, $\Hom_Q(V,V)\cong \mathfrak{gl}_V(\underline{d}(V))$, the isotropy subalgebra of $\mathfrak{gl}(\underline{d}(V))$ at $V$, and we have a natural $\op{Aut}_Q(V)$-equivariant identification of the normal space
$$\Ext_Q(V,V)\cong \Rep(Q,\underline{d}(V))/T_V(\mathcal{O}_V),$$
where $\mathcal{O}_V$ is the orbit of $V$. In particular, $\mathcal{O}_V$ is dense iff $\Ext_Q(V,V)=0$. In this case we say $\underline{d}(V)$ is a \itshape prehomogeneous dimension vector \normalfont and $V$ is the \itshape generic representation\normalfont.

\vspace{0.1in}

We call a polynomial $f\in k[\Rep(Q,\alpha)]$ a semi-invariant, if there is a character $\sigma\in \Hom(\GL(\alpha),k^{\times})$ such that $g\cdot f = \sigma(g) f$. In this case we say the weight of $f$ is $\sigma$. We form the ring of semi-invariants

$$\SI(Q,\alpha)=\bigoplus_{\sigma} \SI(Q,\alpha)_{\sigma}=k[\Rep(Q,\alpha)]^{\SL(\alpha)},$$
where the sum runs over all characters $\sigma$ and the weight spaces are
$$\SI(Q,\alpha)_{\sigma}=\{f\in k[\Rep(Q,\alpha)]\, : \, f \text{ is a semi-invariant of weight } \sigma\}.$$

We investigate the geometry of the \itshape nullcone \normalfont for the action of $\SL(\alpha)$, that is, the the set of common zeros of all semi-invariants of positive degree:

$$\mathcal{Z}(Q,\alpha)=\{X\in \Rep(Q,\alpha)\, : \, f(X)=0, \text{ for all non-constant } f\in \SI(Q,\alpha)\}.$$

We define an important class of determinantal semi-invariants, first constructed by Schofield \cite{scho}. Fix two dimension vectors $\alpha,\beta$, such that $\langle \alpha, \beta \rangle=0$. We define the semi-invariant $c$ of the action of $\GL(\alpha)\times\GL(\beta)$ on $\Rep(Q,\alpha)\times \Rep(Q,\beta)$ by $c(V,W):=\det d^V_W$. Note that we have
$$c(V,W)= 0 \Leftrightarrow \Hom_Q(V,W)\neq 0 \Leftrightarrow \Ext_Q(V,W)\neq 0 $$
Next, for a fixed $V$, restricting $c$ to $\{V\} \times \Rep(Q,\beta)$ defines a semi-invariant $c^V\in \op{SI}(Q,\beta)$. Similarly, for a fixed $W$, restricting $c$ to $\Rep(Q,\alpha)\times \{W\}$, we get a semi-invariant $c_W\in \op{SI}(Q,\alpha)$. 

\vspace{0.1in}

Throughout we assume that $Q$ is without oriented cycles and $\alpha$ is a \itshape prehomogeneous dimension vector\normalfont. Without loss of generality, we assume $\alpha$ is a sincere dimension vector, that is, the dimension at each vertex is positive. Denote by $T$ the generic representation, and write $T=T_1^{\oplus\lambda_1}\oplus \dots \oplus T_m^{\oplus\lambda_m}$, where the $T_i$ are pairwise non-isomorphic direct summands.

\vspace{0.05in}

We denote by $T^\perp$ the right perpendicular category of $T$, that is, the full subcategory of $\Rep(Q)$ consisting of objects $Y$ that satisfy

$$\Hom(T,Y)=\Ext(T,Y)=0.$$

By \cite[Theorem 2.5]{scho}, $T^\perp$ is equivalent to the category of representations of a quiver $Q^\perp$ without oriented cycles and with $n-m$ vertices
$$T^\perp \xrightarrow{\,\sim\,} \Rep(Q^\perp).$$

We denote the simple objects in $T^\perp$ by $S_{m+1},\dots,S_{n}$. We have the following:

\begin{theorem}[{\cite[Theorem 4.3]{scho}}]
\label{thm:schosemi}
The semi-invariants $c_{S_j}$, $j=m+1,\dots, n$, are algebraically independent generators of the ring $\SI(Q,\alpha)$.
\end{theorem}

Hence the nullcone can be described as follows:

$$\mathcal{Z}(Q,\alpha)=\{X\in \Rep(Q,\alpha)\,|\, \Hom(X,S_j)\neq 0, \text{ for all } j=m+1,\dots ,n\}.$$

It is shown in \cite{zwara1} that there is a large enough number $N$, such that if $c\geq N(Q)$ then $\mathcal{Z}(Q,c\cdot \alpha)$ is irreducible and a set-theoretic complete intersection. By \cite{zwara2}, for tame quivers we have more precise control over $N$. Namely, the nullcone is a complete intersection for $N(Q)$ and irreducible for $N(Q)+1$ where

\begin{equation}
N(Q)=\begin{cases}
1 & \text{ if } Q=A_n \text{ or } \widetilde{A}_n,\\
2 & \text{ if } Q=D_n, E_6, E_7 \text{ or } E_8,\\
3 & \text{ if } Q=\widetilde{D}_n, \widetilde{E}_6, \widetilde{E}_7 \text{ or } \widetilde{E}_8.
\end{cases}
\label{eq:bound}
\end{equation}

\vspace{0.1in}

For two representations $M,N\in \Rep(Q,\alpha)$, we say $N$ is a degeneration of $M$ if $N\in \overline{\mathcal{O}(M)}$, where  $\overline{\mathcal{O}(M)}$ is the closure of the orbit of $M$. We say that $N$ is a minimal degeneration of $M$ if for any representation $Q$ such that $N$ is a degeneration of $Q$ and $Q$ is a degeneration of $M$ we have either $Q\cong N$ or $Q\cong M$. By semi-continuity, if $N$ is a degeneration of $M$, then $\dim\Hom_Q(M,X)\leq \dim\Hom_Q(N,X)$ and $\dim\Ext_Q(M,X)\leq \dim\Ext_Q(N,X)$ for any $X\in \Rep(Q).$

It is an easy fact that if we have an exact sequence 
$$0\rightarrow A\rightarrow B\rightarrow C\rightarrow 0$$
then $A\oplus C$ is a degeneration of $B$. In fact for Dynkin quivers we have the following converse by \cite{degext1}:

\begin{lemma}\label{lem:ext}

Let $Q$ be a Dynkin quiver and $M,N\in \Rep(Q)$ such that $N$ be a minimal degeneration of $M$. Then there exists indecomposables $U,V$ such that $N=U\oplus V\oplus X$, $M=Z\oplus X$ and we have an exact sequence
$$0\rightarrow U\rightarrow Z \rightarrow V\rightarrow 0.$$

\end{lemma}

\section{On the reduced property of the nullcone}
\label{sec:redu}

To discuss geometric properties of the nullcone we need to first introduce some tools. We follow much of the notation introduced in \cite{openorb,mater}.

Let $Y$ be a representation satisfying $\Ext(T,Y)=0.$

We denote by $\op{tr} Y$ the trace of $T$ in $Y$, that is, the sum of all the images of all maps from $T$ to $Y$, and let $\overline{Y}=Y/\op{tr} Y$. Then it is easy to see that $\overline{Y}\in T^\perp$.

Next, we recall a construction of Bongartz \cite{bong}. Let $kQ$ be the path algebra of $Q$, viewed as a projective representation of $Q$. Let $\mu_i=\dim \Ext(T_i,kQ)$, $i=1,\dots, m$. Then there is an exact sequence
$$0\rightarrow kQ \rightarrow \tilde{T} \rightarrow T\rightarrow 0$$
such that the induced map
$$\Hom(T_l,\oplus_{i=1}^m T^{\mu_i})\rightarrow \Ext(T_l,kQ)$$
is surjective for all $l=1,\dots m$. This defines $\tilde{T}$ up to isomorphism, and $T\oplus \tilde{T}$ is a tilting module, so it has $n$ pairwise distinct indecomposable summands and $\Ext(T\oplus \tilde{T},T\oplus\tilde{T})=0$. There are $n-m$ non-isomorphic summands of $\tilde{T}$ and denote them by $T_{m+1},\dots,T_n$. We have the following

\begin{proposition}[{\cite{openorb,scho}}]
The representations $\overline{T}_{m+1},\dots,\overline{T}_n$ are representatives for the indecomposable projective objects in $T^\perp$.
\end{proposition}

We order $\overline{T}_{m+1},\dots,\overline{T}_n$ so that they are the projective covers of $S_{m+1},\dots, S_n$, respectively.

For a representation $A=B_1^{b_1}\oplus \dots \oplus B_t^{b_t}$, where $B_i$ are pairwise non-isomorphic indecomposables, and the $b_i$ are positive integers, we denote by $\op{add}(A)$ the full subcategory of $\Rep(Q)$ whose objects are representations $Y$ such that $Y\cong B_1^{c_1}\oplus \dots \oplus B_t^{c_t}$, where the $c_i$ are non-negative integers. 

Now we recall a construction from \cite{openorb}. For every $j=m+1,\dots, n$ we have an exact sequence 
\begin{equation}\label{eq:defin}
0\rightarrow T_j \rightarrow T_j^{++} \rightarrow Z_j \rightarrow 0,
\end{equation}
where the first map is a source map, and $T_j^{++}$ is a representation in $\op{add}(T)$. Denote $Z=T_1\oplus\dots \oplus T_m \oplus Z_{m+1}\dots \oplus Z_n$. We have the following lemma (see \cite{mater, openorb}):

\begin{lemma} The following hold:
\label{lem:tech}
\begin{itemize}
\item[(a)] $Z$ is a tilting module, that is, $\Ext(Z,Z)=0$.
\item[(b)] $\dim \Ext(Z_j,S_k)=\dim \Hom(\overline{T}_j,S_k)=\dim \Hom(T_j,S_k)=\delta_{jk}$, where $j,k\in\{m+1,\dots,n\}$.
\item[(c)] $\Hom(Z,S_k)=0$, where $k\in\{m+1,\dots,n\}$.
\end{itemize}
\end{lemma}

\vspace{0.1in}

\begin{definition} We say $X\in\mathcal{Z}(Q,\alpha)$ satisfies the \itshape independent gradient conditions \normalfont if we have:
\label{jacob}
\begin{itemize}
\item[(a)] $\dim\Hom(X,S_j)=1$, for any $j=m+1,\dots n$,
\item[(b)] For any $k=m+1,\dots,n$ there exists an exact sequence
$$0\rightarrow X\rightarrow Y_k\rightarrow X\rightarrow 0$$
such that $\dim\Hom(Y_k,S_j)=2-\delta_{jk} \text{ for } j=m+1,\dots,n$\,, where $\delta$ is the Kronecker delta.
\end{itemize}
\end{definition}

\vspace{0.1in}

Using \cite[Corollary 7.4]{rankscheme} together with the Serre's Criterion (see for example \cite[Theorem 18.15]{eisen}) we have the following (see also \cite[Proposition 4.6]{maxorb}):

\begin{proposition}
Assume the nullcone $\mathcal{Z}(Q,\alpha)$ is a set-theoretic complete intersection. Then $\mathcal{Z}(Q,\alpha)$ is reduced iff each of its irreducible components contains a representation satisfying the independent gradient conditions (\ref{jacob}).
\end{proposition}

We define the following open subsets of $\mathcal{Z}(Q,\alpha)$: 
$$\mathcal{Z'}(Q,\alpha)=\{X\in \mathcal{Z}(Q,\alpha): \Ext(T,X)=\Ext(X,T)=0\},$$
$$\mathcal{H}(Q,\alpha)=\{X\in \Rep(Q,\alpha): \dim\Hom(X,S_j)=1, \text{ for all } j=m+1,\dots n\}.$$
By the independent gradient conditions (\ref{jacob}), if $\mathcal{H}(Q,\alpha)=\emptyset$, then $\mathcal{Z}(Q,\alpha)$ is not reduced. Now we are ready to prove our first result about reduced property of the nullcone:

\begin{proposition}\label{prop:reduced}
Assume that $\mathcal{Z'}(Q,\alpha)$ is not empty. Then the nullcone $\mathcal{Z}(Q,\alpha)$ is reduced, irreducible and a complete intersection.
\end{proposition}

\begin{proof}
By \cite[Proposition 3.7]{zwara1}, we know already that $\mathcal{Z}(Q,\alpha)$ is irreducible and a complete intersection. 

First, we show that he set $\mathcal{H}(Q,\alpha)$ is non-empty. Assuming the contrary, for all $X\in \mathcal{Z}(Q,\alpha)$ we have that $\dim \Hom(X,S_{m+1}\oplus\cdots\oplus S_n) > m-n$. Take the following subset of $\mathcal{Z}(Q,\alpha)$:
$$\mathcal{A}=\{X\in \Rep(Q,\alpha): \exists \text{ epimorphism } X\twoheadrightarrow S_{m+1}\oplus\cdots\oplus S_n\}.$$
By \cite[Lemma 3.6]{zwara1} we have $\mathcal{Z'}(Q,\alpha)\subset\mathcal{A}$. In particular, $\op{codim} \mathcal{A}=n-m$. But since $\dim \Hom(X,S_{m+1}\oplus\cdots\oplus S_n) > m-n$, for all $X\in \mathcal{A}$, the proof of \cite[Lemma 3.1]{zwara1} implies that $\op{codim} \mathcal{A}>m-n$, a contradiction. Hence $\mathcal{H}(Q,\alpha)$ is non-empty.

Now $\mathcal{Z'}(Q,\alpha)\cap \mathcal{H}(Q,\alpha)$ is open and non-empty, and we prove that any of its elements satisfy the independence condition (\ref{jacob} b). So take an arbitrary $X\in \mathcal{Z'}(Q,\alpha)\cap \mathcal{H}(Q,\alpha)$ and write 
$$X\cong \widetilde{X} \oplus \bigoplus_{j=m+1}^n Z_j^{a_j},$$
such that $\widetilde{X}$ and $Z$ have no common indecomposable summands. As in \cite[Propositions 3.14, 3.15]{mater}, we have a minimal projective resolution of $\overline{X}$ in $T^\perp$ of the form
$$0\rightarrow \bigoplus_{j\in J} \overline{T}_j \rightarrow \bigoplus_{j=m+1}^n \overline{T}_j \rightarrow \overline{X}\rightarrow 0,$$
where $J\subset \{m+1,\dots,n\}$. Moreover, $a_j=0$ if $j\in J$, and $a_j=1$, if $j\in J^c$, where $J^c$ denotes the complement of $J$ in $\{m+1,\dots,n\}$. We construct the exact sequences as in (\ref{jacob} b) by considering two cases, whether $k\in J$ or $k\in J^c$.

First, let $k\in J^c$.  Consider the composite map $\phi:T_k\to\overline{T}_k\to \overline{X}$ where the first map is the projection $T_k\to T_k/\op{tr} T_k$ and the second is from the minimal resolution of $\overline{X}$. Applying $\Hom(-,S_k)$ to the minimal resolution, together with Lemma \ref{lem:tech} we have the induced isomorphisms of $1$-dimensional spaces
\begin{equation}\label{eq:isomo}
\Hom(\overline{X},S_k)\cong \Hom(\overline{T}_k,S_k)\cong \Hom(T_k,S_k)\cong \Ext(Z_k,S_k).
\end{equation} 
Consider the following diagram
\[\xymatrix{\
0 \ar[r] & T_k \ar[r]\ar[d]^{\phi} & T_k^{++} \ar[r]\ar[d] & Z_k \ar[r]\ar@{=}[d] & 0\\
0 \ar[r] & \overline{X} \ar[r] & U_k \ar[r] & Z_k \ar[r] & 0
}\]
where the second row is the push-out of the first via $\phi$. Take $j\in \{m+1,\dots, n\}, j\neq k$. Applying $\Hom(-,S_j)$ to the second exact sequence, the induced long exact sequence together with Lemma \ref{lem:tech} gives 
$$\dim \Hom(U_k,S_j)=\dim\Hom(\overline{X},S_j) = 1.$$
On the other hand, applying $\Hom(-,S_k)$ we get the exact sequence
$$0\rightarrow \Hom(U_k,S_k) \rightarrow \Hom(\overline{X},S_k) \rightarrow \Ext(Z_k,S_k),$$
where, by construction, the last map is the composition of isomorphisms in (\ref{eq:isomo}). Hence $\Hom(U_k,S_k)=0$, so we have $\dim \Hom(U_k,S_j)=\delta_{jk}$, where $j\in \{m+1,\dots,n\}$. 

Now applying $\Hom(Z_k,-)$ to the exact sequence 
$$0\rightarrow \op{tr} X \rightarrow X \rightarrow \overline{X}\rightarrow 0$$
we get $\Ext(Z_k,X)\cong \Ext(Z_k,\overline{X})$, since $\op{tr} X \in \op{add}(Z)$ by \cite[Proposition 3.9]{mater}. Hence can we lift (uniquely) the exact sequence with middle term $U_k$ and get the following exact diagram:

\[\xymatrix@R-0.5pc{\
 & 0 \ar[d] & 0 \ar[d] &  &\\
 & \op{tr} X \ar[d]\ar@{=}[r] & \op{tr} X\ar[d] &  & \\
0 \ar[r] & X \ar[r]\ar[d] & V_k \ar[r]\ar[d] & Z_k \ar[r]\ar@{=}[d] & 0\\
0 \ar[r] & \overline{X} \ar[d]\ar[r] & U_k \ar[d] \ar[r] & Z_k\ar[r]& 0\\
& 0 & 0 &  &
}
\]

Applying $\Hom(-,S_j)$ for $j=m+1,\dots,n$ to the middle column we get that $\dim \Hom(V_k,S_j) = \dim \Hom(U_k,S_j)= \delta_{jk}$. By Lemma \ref{lem:tech} $\Hom(Z,S_j)=0$, so $\dim\Hom(X,S_j)=\dim\Hom(\widetilde{X},S_j)=1$.  Hence if we put 
$$Y_k=V_k\oplus \widetilde{X} \oplus \bigoplus_{j\in J^c\backslash \{k\}} Z_j$$
we have an exact sequence 
$$0\rightarrow X \rightarrow Y_k\rightarrow X \rightarrow 0$$
satisfying the independence condition (\ref{jacob} b).

\vspace{0.05in}

Now we consider the second case, when $k\in J$. Denote

\[Q_k=\bigoplus_{j\in J\backslash\{k\}}\overline{T}_j \text{ and } R_k= \bigoplus_{\substack{j=m+1\\j\neq k}}^n\overline{T}_j.\]

Let $\psi$ denote the injective map of the minimal projective resolution of $\overline{X}$ in $T^\perp$, so $\psi:Q_k\oplus \overline{T}_k \to R_k \oplus \overline{T}_k$. Consider the following commutative diagram
\[\xymatrix{
0 \ar[r] & Q_k \oplus \overline{T}_k \ar[r]^-{\binom{I}{0}}\ar[d]^{\psi} & Q_k \oplus \overline{T}_k \oplus Q_k \oplus \overline{T}_k \ar[r]^-{(0 ~ I)}\ar[d]^{\psi_k} & Q_k\oplus \overline{T}_k \ar[r]\ar[d]^{\psi} & 0\\
0 \ar[r] & R_k \oplus \overline{T}_k \ar[r]^-{\binom{I}{0}} & R_k \oplus \overline{T}_k \oplus R_k \oplus \overline{T}_k \ar[r]^-{(0 ~ I)} & R_k\oplus \overline{T}_k \ar[r] & 0
}\]
where the map $\psi_k:Q_k^2\oplus \overline{T}_k^2 \to R_k^2\oplus \overline{T}_k^2$ is obtained from $\psi\oplus \psi$ by adding the identity map from the second copy of $\overline{T}_k$ to the first copy of $\overline{T}_k$. Denote $W_k=\op{coker} \psi_k$. By the snake lemma, $\psi_k$ is injective and we have an exact sequence
\[0\rightarrow \overline{X} \rightarrow W_k\rightarrow \overline{X} \rightarrow 0.\]
Moreover, applying $\Hom(-,S_j)$ to the (non-minimal) projective resolution $\psi_k:Q_k^2\oplus \overline{T}_k^2 \to R_k^2\oplus \overline{T}_k^2$, we have that $\dim \Hom(W_k,S_j)=2-\delta_{jk}$, for $j=m+1,\dots, n$. Now we pull-back the sequence above via the map $X \twoheadrightarrow \overline{X}$ to get the following diagram
\[\xymatrix@R-0.5pc{
 & & 0 \ar[d] & 0 \ar[d] &\\
 & & \op{tr} X \ar[d]\ar@{=}[r] & \op{tr} X \ar[d] & \\
 0 \ar[r] & \overline{X} \ar[r]\ar@{=}[d] & W'_k \ar[r]\ar[d] & X \ar[r]\ar[d] & 0 \\
 0 \ar[r] & \overline{X} \ar[r] & W_k \ar[r]\ar[d] & \overline{X} \ar[r]\ar[d] & 0 \\
 & & 0 & 0 &
}\]
Applying $\Hom(-,S_j)$ to the middle column, we see that $\dim \Hom(W'_k, S_j)=\dim\Hom(W_k,S_j)=2-\delta_{jk}$, where $j=m+1,\dots, n$. 

Now the surjection $X \twoheadrightarrow \overline{X}$ gives a surjective map $\Ext(X,X) \twoheadrightarrow \Ext(X,\overline{X})$, hence we can lift the sequence $0\to \overline{X} \to W'_k \to X \to 0$ to a sequence $0\to X\to Y_k \to X \to 0$. As before, from the sequence $0\to \op{tr} X \to Y_k\to W'_k\to 0$ we see that $\dim \Hom(Y_k, S_j)=\dim\Hom(W'_k,S_j)=2-\delta_{jk}$, where $j=m+1,\dots, n$, giving the desired property.
\end{proof} 

\vspace{0.1in}

We can conclude a fortiori that for a reduced nullcone we have:

\begin{corollary}
The set $\mathcal{H}(Q,\alpha)\cap \mathcal{Z'}(Q,\alpha)$ is contained in the smooth locus of $\mathcal{Z}(Q,\alpha)$, which in turn is contained in $\mathcal{H}(Q,\alpha)$.
\end{corollary}

We also deduce the following result:
 
\begin{theorem}
\label{thm:reduc}
Let $T_1,\dots, T_m$ be pairwise non-isomorphic indecomposables in $\Rep(Q)$ such that $\Ext(T_i,T_j)=0$, for any $i,j\leq m$. Then there is a positive integer $N$ such the nullcone $\mathcal{Z}(Q,\alpha)$ is reduced, irreducible and a complete intersection for any dimension vector $\alpha = \lambda_1\cdot \underline{d}(T_1) + \dots + \lambda_m \cdot\underline{d}(T_m)$ with $\lambda_i\geq N$ for $i=1,\dots, m$.
\end{theorem}

\begin{proof}
If we pick $N$ to be large enough, the set $\mathcal{Z'}(Q,\alpha)$ is not empty by \cite[Corollary 3.4]{zwara1} or by \cite[Proposition 4.7]{mater}. 
\end{proof}

\begin{remark}
One can give the following short proof of the theorem above. If we allow $N$ to be large enough so that $\alpha - \underline{d}(T_{m+1}^{++})-\dots -  \underline{d}(T_n^{++})$ is a dimension vector, then there is an element of the form $Z_{m+1}\oplus T_{m+1} \oplus \dots \oplus Z_n \oplus T_n \oplus T'$ that lies in $\mathcal{H}(Q,\alpha)\cap \mathcal{Z'}(Q,\alpha)$, where $T'\in \op{add} T$. Moreover, one can easily see that the independence condition (\ref{jacob} b) is also satisfied by using the just the sequences (\ref{eq:defin}). However, this condition on $N$ is cruder than the condition $\mathcal{Z'}(Q,\alpha)\neq \emptyset$.
\end{remark}

\begin{remark}\label{characteristic}
Hence we can conclude that if $\alpha$ is not \lq\lq too small", then the semi-invariants $c_{S_{m+1}},\dots, c_{S_n}$ form a regular sequence and generate a prime ideal in $k[\Rep(Q,\alpha)]$. In fact, these properties hold for an arbitrary field $k$ (being geometrically reduced and irreducible). This is because the semi-invariants $c_{S_{m+1}},\dots, c_{S_n}$ are defined over any field $k$ (not necessarily algebraically closed) by construction, since the representations $S_i$  themselves are ($\dim S_k$ are real Schur roots), cf. \cite{kac,schofieldef}.
\end{remark}

For tame quivers, one can give more precise information on $N$. Bounds for a condition similar to $\mathcal{Z'}(Q,\alpha)\neq \emptyset$ have been investigated previously in \cite{mater}. So for Dynkin quivers $N$ can be taken to be $N(Q)+1$ as in (\ref{eq:bound}). Also, for extended Dynkin quivers similar bounds have been announced in \cite[Remark 6.7]{mater2}. 

However, for Dynkin quivers we show by a different reasoning that for the nullcone to be reduced we only need $N=N(Q)$. We keep the usual notation.

\begin{theorem}
\label{thm:dynk}
Let $Q$ we a Dynkin quiver and set $N(Q)=1$, if $Q$ is of type $A$, and $N(Q)=2$ otherwise. If $\lambda_i\geq N(Q)$ for all $i=1,\dots,m$\,, then the nullcone $\mathcal{Z}(Q,\alpha)$ is reduced and a complete intersection.
\end{theorem}

\begin{proof}
The bounds from (\ref{eq:bound}) imply that the nullcone is a set-theoretic complete intersection. Hence it is enough to verify the independent gradient conditions (\ref{jacob}). We are going to prove the result for arbitrary zero sets of semi-invariants $Z(c_{S_{i_1}},\dots ,c_{S_{i_k}})$, $i_j\in\{m+1,\dots,n\}$, for $j=1,\dots,k$\,, by induction on the number of semi-invariants $k\leq n-m$. If $k=0$ there is nothing to prove. For simplicity, we can assume $i_j=j$ and denote $f_j=c_{S_j}$, $j=1,\dots,k$. Now take any irreducible component of $\mathcal{Z}:=Z(f_1,\dots ,f_k)$, which, since $Q$ is of finite type, is the closure of the orbit of a representation, say $X$. Take $l=1,\dots, k$ arbitrary and look at the zero-set $\mathcal{Z}_l=Z(f_1,\dots, f_{l-1},f_{l+1},\dots f_k)$. Since the zero-sets are complete intersections, $\dim \mathcal{Z}_l=\dim \mathcal{Z}+1$. Hence there is an irreducible component of $\mathcal{Z}_l$ which is the closure of an orbit, say $X_l$, so that $X$ is a minimal degeneration of $X_l$. By Lemma \ref{lem:ext} we can write
\[X\cong A_l\oplus B_l \oplus Y_l \text{ and } X_l\cong C_l \oplus Y_l\]
such that $A_l,B_l$ are indecomposables and there is an exact sequence $0\to A_l \to C_l \to B_l\to 0$. By the induction hypothesis, $X_l$ satisfies the independent gradient condition (\ref{jacob} a), hence $\dim \Hom(X_l,S_j)=\Ext(X_l,S_j)=1-\delta_{jl}$, for $j=1,\dots, k$. This implies that $\Hom(Y_l,S_l)=\Hom(C_l,S_l)=\Hom(B_l,S_l)=\Ext(Y_l,S_l)=\Ext(C_l,S_l)=\Ext(A_l,S_l)=0$, $\dim\Hom(X,S_l)=\dim\Hom(A_l,S_l)=\dim\Ext(B_l,S_l)>0$.

\vspace{0.05in}

Now assume that $X$ does not satisfy the first independent gradient condition (\ref{jacob} a), hence we assume WLOG that $\dim \Hom(X,S_1)>1$. Then $\dim \Hom(A_1,S_1)=\Ext(B_1,S_1)>1$, hence $A_1$ and $B_1$ are not direct summands of $X_l$ or $Y_l$, for any $l=1,\dots, k$. Suppose $A_1\cong B_l$, for some $l>1$. Then the sequence $0\to A_l \to C_l \to B_l\to 0$ gives $\dim\Hom(C_l,S_1)=\dim\Hom(A_1,S_1)>1$, which again is a contradiction. Hence $A_l\cong A_1$, $B_l\cong B_1$, hence also $Y_l\cong Y_1$ for all $l=1,\dots, k$. So put $A:=A_1$, $B:=B_1$, $Y=Y_1$. Summarizing, $X$ must be of the form 
\[X\cong A\oplus B\oplus Y,\]
with $A,B$ indecomposables together with exact sequences $0\to A\to C_l \to B\to 0$, for $l=1,\dots,k$. 

Now we also see that $\Ext(A,Y)=\Ext(Y,A)=\Ext(B,Y)=\Ext(Y,B)=\Ext(Y,Y)=\Ext(A,B)=0$. Indeed, suppose, for example, that we have a non-trivial exact sequence $0\to A \to U \to Y\to 0$. Then $X$ is a degeneration of $U\oplus B$, hence $U\oplus B$ is not an element in $\mathcal{Z}$. But $\dim\Ext(U\oplus B,S_j)= \dim\Ext(B,S_j)>0$, for any $j=1,\dots,k$, which implies $U\oplus B \in \mathcal{Z}$, a contradiction. This proves $\Ext(Y,A)=0$, and the other claims are analogous. Summarizing, we have $k=\dim\Ext(X,X)=\dim\Ext(B,A)>0$.

Next, we claim that $Y\in \op{add}(T)$. Since $Q$ is Dynkin, using that an indecomposable has no self-extensions, we see from Lemma \ref{lem:ext} that we can reach $T$ from $X$ by a sequence of minimal degenerations given by short exact sequences. However, since $\Ext(A,Y)=\Ext(Y,A)=\Ext(B,Y)=\Ext(Y,B)=\Ext(Y,Y)=0$, $Y$ is going to remain fixed in this sequence of short exact sequences. Hence when we reach $T$, we get that $Y$ is a direct summand of $T$, so we can write $Y=T_1^{\beta_1}\oplus\dots\oplus T_m^{\beta_m}$, with $0\leq \beta_i\leq \lambda_i$, for $i=1,\dots,m$. Also $\underline{d}(A)+\underline{d}(B)=(\lambda_1-\beta_1)\underline{d}(T_1)+\dots+(\lambda_m-\beta_m)\underline{d}(T_m)$.

If $Q$ is of type $A_n$, the dimension of the space of maps between any two indecomposables is at most $1$, contradicting $\dim\Hom(A,S_1)>1$. Hence we may assume $N(Q)=2$. Then we claim that $\beta_i\geq 1$, for all $i=1,\dots,m$. Assume the contrary, say $\beta_1=0$. Then $\langle \underline{d}(A)+\underline{d}(B),\underline{d}(A)+\underline{d}(B)\rangle \geq \lambda_1^2 \dim\Hom(T_1,T_1)\geq 4$. On the other hand, we also have $\langle \underline{d}(A)+\underline{d}(B),\underline{d}(A)+\underline{d}(B)\rangle=2+\langle \underline{d}(A),\underline{d}(B)\rangle+\langle \underline{d}(B),\underline{d}(A)\rangle$, hence the inequality $\langle \underline{d}(A),\underline{d}(B)\rangle+\langle\underline{d}(B),\underline{d}(A)\rangle\geq 2$. We claim that this is impossible for Dynkin quivers. Indeed, we can use reflection functors (see \cite{elements}) to reduce to the case when $A\cong S^x$ is a simple corresponding to a vertex $x$. Then $B$ is not isomorphic to the simple $S^x$, and the value becomes $\langle \underline{d}(S^x),\underline{d}(B)\rangle+\langle \underline{d}(B),\underline{d}(S^x)\rangle= 2\dim B(x)-\dim B(y_1)-\dots - \dim B(y_i)$, where $y_i$ are all the neighbors of $x$. But for Dynkin quivers this value is known to be smaller than $2$ (see \cite[(4) on Page 4]{quad}). Hence $\beta_i\geq 1$, for all $i=1,\dots,m$. Since $\Ext(A,Y)=\Ext(Y,A)=\Ext(B,Y)=\Ext(Y,B)=0$, this implies that $X\in \mathcal{Z'}(Q,\alpha)$, which is a contradiction, by Proposition \ref{prop:reduced}. Hence $X$ satisfies the first independent gradient condition (\ref{jacob} a), that is, $\dim(X,S_j)=1$, for $j=1,\dots,k$.

Now we show that $X$ satisfies the independence condition (\ref{jacob} b). Take the exact sequences $0\to A_l \to C_l \to B_l\to 0$ as before. These induce sequences 
\[0\rightarrow X\rightarrow X\oplus X_l\rightarrow X\rightarrow 0\]
that satisfy the second indenpendent gradient condition, since $\dim \Hom(X\oplus X_l,S_j)=2-\delta_{jl}$, where $j=1,\dots, k$. Hence each irreducible component of $\mathcal{Z}$ is reduced, finishing the inductive step.
\end{proof}

\vspace{0.05in}

\begin{remark}
We note that for type $A$ quivers the above result also follows from \cite{ryan}, as the fact that (in characteristic $0$) the nullcone has rational singularities (see Section \ref{sec:bern}).
\end{remark}

\vspace{0.05in}

The following example shows that the nullcone of a Dynkin quiver is not always reduced.

\vspace{0.05in}

\begin{example}\label{ex:notred}
Let $Q$ be $E_8$ with the following orientation and dimension vector:

\[\xymatrix{
& & 3 \ar[d] & & & &\\
2 \ar[r] & 4 \ar[r] & 7 & 4 \ar[l] & 3 \ar[l] & 2 \ar[l] &1 \ar[l]
}\]

The decomposition of the generic representation is as follows: 

\[\arraycolsep=1.5pt
\begin{array}{ccccccc}
 &  & 3 & & & & \\
2&4 & 7 &4&3&2&1 
\end{array}=
\begin{array}{ccccccc}
 &  & 1 & & & & \\
0&1 & 2 &1&1&1&1 
\end{array}
\oplus
\begin{array}{ccccccc}
 &  & 1 & & & & \\
1&2 & 3 &2&1&1&0 
\end{array}
\oplus
\begin{array}{ccccccc}
 &  & 1 & & & & \\
1&1 & 2 &1&1&0&0 
\end{array}
\]

The simples in the right orthogonal category are the following:

\[\arraycolsep=1.5pt
\begin{array}{ccccccc}
 &  & 0 & & & & \\
0&0 & 1 &1&1&1&1 
\end{array}\,,\,
\begin{array}{ccccccc}
 &  & 1 & & & & \\
0&1 & 2 &1&1&1&0 
\end{array}\,,\,
\begin{array}{ccccccc}
 &  & 0 & & & & \\
1&1 & 1 &0&0&0&0 
\end{array}\,,\,
\begin{array}{ccccccc}
 &  & 1 & & & & \\
0&0 & 1 &1&0&0&0 
\end{array}\,,\,
\begin{array}{ccccccc}
 &  & 0 & & & & \\
0&1 & 1 &1&1&0&0 
\end{array}\]

A routine computation shows that the nullcone consists of $9$ irreducible components, each the closure of one of the following representations:
\[\begin{aligned}
N_1 &= \arraycolsep=1.5pt \begin{array}{ccccccc}
 & &0& & & & \\
0&0&1&0&0&0&0 
\end{array}
\oplus
\begin{array}{ccccccc}
 & &3& & & & \\
2&4&6&4&3&2&1 
\end{array}\\
N_2 &= \arraycolsep=1.5pt \begin{array}{ccccccc}
 & &1& & & & \\
0&0&1&1&0&0&0 
\end{array}
\oplus
\begin{array}{ccccccc}
 & &0& & & & \\
0&0&1&1&1&0&0 
\end{array}
\oplus
\begin{array}{ccccccc}
 & &1& & & & \\
0&1&1&0&0&0&0 
\end{array}
\oplus
\begin{array}{ccccccc}
 & &0& & & & \\
1&1&1&0&0&0&0 
\end{array}
\oplus
\begin{array}{ccccccc}
 & &1& & & & \\
1&2&3&2&2&2&1 
\end{array}\\
N_3 &= \arraycolsep=1.5pt \begin{array}{ccccccc}
 & &0& & & & \\
1&1&1&0&0&0&0 
\end{array}
\oplus
\begin{array}{ccccccc}
 & &0& & & & \\
0&0&1&1&1&1&1 
\end{array}
\oplus
\begin{array}{ccccccc}
 & &1& & & & \\
0&1&2&1&0&0&0 
\end{array}
\oplus
\begin{array}{ccccccc}
 & &2& & & & \\
1&2&3&2&2&1&0 
\end{array}\\
N_4 &= \arraycolsep=1.5pt \begin{array}{ccccccc}
 & &1& & & & \\
0&0&1&1&0&0&0 
\end{array}
\oplus
\begin{array}{ccccccc}
 & &0& & & & \\
0&0&1&1&1&1&1 
\end{array}
\oplus
\begin{array}{ccccccc}
 & &0& & & & \\
0&1&1&0&0&0&0 
\end{array}
\oplus
\begin{array}{ccccccc}
 & &1& & & & \\
1&1&2&1&1&1&0 
\end{array}
\oplus
\begin{array}{ccccccc}
 & &1& & & & \\
1&2&2&1&1&0&0 
\end{array}\\
N_5 &= \arraycolsep=1.5pt \begin{array}{ccccccc}
 & &1& & & & \\
0&0&1&0&0&0&0 
\end{array}
\oplus
\begin{array}{ccccccc}
 & &0& & & & \\
0&0&1&1&1&1&1 
\end{array}
\oplus
\begin{array}{ccccccc}
 & &0& & & & \\
0&1&1&1&1&0&0 
\end{array}
\oplus
\begin{array}{ccccccc}
 & &0& & & & \\
1&1&1&0&0&0&0 
\end{array}
\oplus
\begin{array}{ccccccc}
 & &2& & & & \\
1&2&3&2&1&1&0 
\end{array}\\
N_6 &= \arraycolsep=1.5pt \begin{array}{ccccccc}
 & &1& & & & \\
0&0&1&1&0&0&0 
\end{array}
\oplus
\begin{array}{ccccccc}
 & &0& & & & \\
0&0&1&0&0&0&0 
\end{array}
\oplus
\begin{array}{ccccccc}
 & &0& & & & \\
0&0&1&1&1&1&0 
\end{array}
\oplus
\begin{array}{ccccccc}
 & &1& & & & \\
0&1&1&0&0&0&0 
\end{array}
\oplus
\begin{array}{ccccccc}
 & &0& & & & \\
1&1&1&0&0&0&0 
\end{array}
\oplus
\begin{array}{ccccccc}
 & &1& & & & \\
1&1&2&1&1&1&1 
\end{array}\\
N_7 &= \arraycolsep=1.5pt \begin{array}{ccccccc}
 & &1& & & & \\
0&0&1&1&0&0&0 
\end{array}
\oplus
\begin{array}{ccccccc}
 & &0& & & & \\
0&0&1&1&1&1&1 
\end{array}
\oplus
\begin{array}{ccccccc}
 & &0& & & & \\
0&1&1&1&1&0&0 
\end{array}
\oplus
\begin{array}{ccccccc}
 & &1& & & & \\
0&1&2&1&1&1&0 
\end{array}
\oplus
\begin{array}{ccccccc}
 & &0& & & & \\
1&1&1&0&0&0&0 
\end{array}
\oplus
\begin{array}{ccccccc}
 & &1& & & & \\
1&1&1&0&0&0&0 
\end{array}\\
N_8 &= \arraycolsep=1.5pt \begin{array}{ccccccc}
 & &0& & & & \\
1&1&1&0&0&0&0 
\end{array}
\oplus
\begin{array}{ccccccc}
 & &0& & & & \\
0&0&1&1&0&0&0 
\end{array}
\oplus
\begin{array}{ccccccc}
 & &2& & & & \\
1&2&4&3&3&2&1 
\end{array}
\oplus
\begin{array}{ccccccc}
 & &1& & & & \\
0&1&1&0&0&0&0 
\end{array}\\
N_9 &= \arraycolsep=1.5pt \begin{array}{ccccccc}
 & &1& & & & \\
0&0&1&1&0&0&0 
\end{array}
\oplus
\begin{array}{ccccccc}
 & &0& & & & \\
0&0&1&1&1&1&1 
\end{array}
\oplus
\begin{array}{ccccccc}
 & &1& & & & \\
0&1&2&1&1&0&0 
\end{array}
\oplus
\begin{array}{ccccccc}
 & &0& & & & \\
1&1&1&0&0&0&0 
\end{array}
\oplus
\begin{array}{ccccccc}
 & &1& & & & \\
1&2&2&1&1&1&0 
\end{array}
\end{aligned}\]

If we look at the first component $N_1$, we see that indeed the nullcone is not reduced, since 
\[\arraycolsep=1.5pt \dim \Hom\left( \begin{array}{ccccccc}
 & &0& & & & \\
0&0&1&0&0&0&0 
\end{array}\,,\,
\begin{array}{ccccccc}
 &  & 1 & & & & \\
0&1 & 2 &1&1&1&0 
\end{array}\right) = 2,
\]
contradicting the independent gradient condition (\ref{jacob} a). We also note that the nullcone is a set-theoretic complete intersection, since the codimension of each component is $\dim \Ext(N_i,N_i)=5$, for all $i=1,\dots, 9$.
\end{example}

\vspace{0.05in}

\section{Bernstein-Sato polynomials and rational singularities of zero sets}
\label{sec:bern}

In this section, we will work over the complex field $k=\complex$. 
Let $V$ be a vector space, $k[V]$ the polynomial ring and $\mathcal{D}_V$ the ring of (polynomial) differential operators on $V$. 


First, we have recall the notion of Bernstein-Sato polynomials for ideals (see \cite{bms}). 

Consider polynomials $f_1,\dots, f_r$ in $k[V]$ and the ideal $I=(f_1,\dots,f_r)$. Denote by $\underline{e}^1,\dots, \underline{e}^r$ the standard basis for $\integ^r$, and put $\underline{e}=\underline{e}^1+\dots+\underline{e}^r$. For $\underline{c}=(c_1,\dots, c_r)\in \integ^{r}$, put $I(\underline{c})^{-}=\{i : c_i < 0\}$ and $I(\underline{c})^{+}=\{i : c_i > 0\}$. Let $s_1,\dots, s_r$ be independent variables, and put $s=s_1+\dots+s_r$. 
The Bernstein-Sato polynomial $b(s)$ is the monic polynomial of smallest degree such that $b(s)\prod_i f^{s_i}$ belongs to the $\mathcal{D}_V[s_1,\dots,s_r]$-submodule generated by the elements
\begin{equation}
\prod_{i\in I(\underline{c})^-} \binom{s_i}{-c_i}\prod_{i=1}^r f_i^{s_i+c_i},
\label{eq:bern}
\end{equation}
where $\underline{c}$ runs over elements in $\integ^r$ with dot product $\underline{e}\cdot\underline{c}=1$, and $\binom{s_i}{m}=s_i (s_i-1)\cdots (s_i-m+1)/m!$.


The roots of $b(s)$ are negative and rational. Let $Z$ be the (not necessarily reduced) variety defined by $I$. We have the following (see \cite{saito} for the hypersurface case):
\begin{theorem}[{\cite[Theorem 4]{bms}}]\label{thm:ratio}
Assume $Z$ is a reduced complete intersection of codimension $r$. Then $Z$ has rational singularities if and only if $-r$ is the largest root of $b(s)$ and its multiplicity is $1$.
\end{theorem}

Now we recall the $b$-function of several variables. Assume $G$ is a reductive group and $V$ is a prehomogeneous vector space under the action of $G$, that is, there exists a dense orbit. Assume $f_1,\dots,f_r$ are semi-invariants, and let $f_1^*,\dots , f_r^*$ their respective dual semi-invariants (of opposite weight) viewed as differential operators in $\mathcal{D}_V$ (see \cite[Lemma 1.5]{gyoja}). Put $\underline{s}=(s_1,\dots ,s_r)$ independent variables as above. For any $r$-tuple $\underline{m}=(m_1,\dots,m_r)$ of non-negative integers, the $b$-function (of several variables) is the polynomial $b_{\underline{m}}(\underline{s})$ obtained by
\begin{equation}
\prod_{i=1}^r f_i^{*m_i}\cdot \prod_{i=1}^r f_i^{s_i+m_i} =b_{\underline{m}}(\underline{s})\prod_{i=1}^r f_i^{s_i}.
\label{eq:bfun}
\end{equation}

It is known that $b_{\underline{m}}(\underline{s})$ is a product of linear polynomials (\cite{sata, ukai}) and has an expression (up to a constant) of the form:

\begin{equation}
b_{\underline{m}}(\underline{s})= \prod_{j=1}^N \prod_{k=1}^{\mu_j}\prod_{i=0}^{\gamma^j\cdot \underline{m}-1} (\gamma^j\cdot \underline{s} + \alpha_{j,k}+i),
\label{eq:form}
\end{equation}

where $N\in \nat$, $\mu_j\in \nat$, $\gamma^j\in \nat^r$ and $\alpha_{j,k}\in \rat_+$.

\vspace{0.05in}

The first result of this section is a link between $b_{\underline{m}}(\underline{s})$ and $b(s)$, beyond the case $r=1$ when the two notions coincide (cf. \cite{gyoja}). For $\underline{c}\in \integ^r$, define $\underline{c}^{+}\in \nat^r$ by $\underline{c}^{+}_i:= \max\{c_i,0\}$ and $\underline{c}^{-}:=\underline{c}-\underline{c}^{+}$, and denote by $b_{\underline{c}}\in k[s_1,\dots,s_r]$ the polynomial
$$b_{\underline{c}}:=b_{\underline{c}^+}(\underline{s}+\underline{c}^{-})\cdot \prod_{i\in I(\underline{c})^-} \binom{s_i}{-c_i}.$$

Let $\tilde{B}$ be the ideal generated by polynomials $b_{\underline{c}}$, where $\underline{c}$ runs over the elements in $\integ^r$ with $\underline{e}\cdot \underline{c}=1$.

\begin{proposition}
\label{prop:link}
There exists a polynomial in $\tilde{B}$ depending only in $s=s_1+\dots+s_r$; let $\tilde{b}(s)$ be such of lowest degree. We have $b(s) \, | \, \tilde{b}(s)$.
\end{proposition}

\begin{proof}
The second part of the proposition follows from
$$\prod_{i\in I(\underline{c})^-}^r f_i^{-c_i}\prod_{i\in I(\underline{c})^+}^r f_i^{*c_i}\cdot \prod_{i\in I(\underline{c})^-} \binom{s_i}{-c_i}\prod_{i=1}^r f_i^{s_i+c_i} = b_{\underline{c}} \cdot \prod_{i=1}^r f_i^{s_i}.$$

Put $L=\{\gamma^1,\dots, \gamma^N\}\cup \{\underline{e}^1,\dots,\underline{e}^r\}$. Choose arbitrary elements $\underline{l}^1,\dots, \underline{l}^k\in L$ and $u_1,\dots, u_k\in \rat$ and take the ideal $I=(\underline{l}^1\cdot\underline{s}+u_1,\dots,\underline{l}^k\cdot \underline{s}+u_k)$, and assume $I$ is a proper ideal. Since $\tilde{B}$ is finitely generated, in order to prove the first part of the proposition, it is enough to show that if $\tilde{B}\subset I$, then $\underline{e}\in\op{span}_{\rat} \{\underline{l}^1,\dots, \underline{l}^k\}$. WLOG, we can assume that for any $\underline{l}\in L$ with $\underline{l}\in\op{span}_{\rat}\{\underline{l}^1,\dots,\underline{l}^k\}$ we have $\underline{l}\in\{\underline{l}^1,\dots,\underline{l}^k\}$. Arguing by induction on $r$, we can further assume WLOG that there are no basis elements $\underline{e}^i$ among the vectors $\underline{l}^1,\dots, \underline{l}^k$. Then, we show that in fact $\underline{e}\in\op{span}_{\rat_+} \{\underline{l}^1,\dots, \underline{l}^k\}$. Assuming the contrary, there exists by Farkas' lemma a vector $\underline{c}\in \integ^r$ such that $\underline{e}\cdot \underline{c}>0$, and $\underline{l}^i \cdot \underline{c}<0$, for all $i=1,\dots, k$. Since $\underline{l}^i \in \nat^r$, we can, by possibly scaling and decreasing the entries of $\underline{c}$, find a vector $\underline{c}\in \integ^r$ with $\underline{e}\cdot \underline{c}=1$ and $\max_i \{\underline{l}^i \cdot \underline{c} \}$ arbitrary small. Hence, looking at largest constant terms of the factors in (\ref{eq:form}) we can find $\underline{c}$ such that none of the forms $\underline{l}^1\cdot\underline{s}+u_1,\dots,\underline{l}^k\cdot \underline{s}+u_k$ is a factor of $b_{\underline{c}}$, which gives $b_{\underline{c}}\notin I$, a contradiction.
\end{proof}

\begin{remark}
The $b$-function of several variables (\ref{eq:bfun}) has been generalized to the case of arbitrary (not necessarily semi-invariant) polynomials by \cite{gyoja2,sab}. Proposition \ref{prop:link} can be adapted to this setting as well. In particular, this gives another proof for the existence of $b(s)$ and rationality of its roots (see \cite{bms}). However, in general $\tilde{b}(s)\neq b(s)$, moreover $\tilde{b}(s)$ may have positive roots (see Example \ref{ex:pos}).
\end{remark}

\vspace{0.05in}

Now, as in the previous section, let $Q$ be a quiver with $n$ vertices, $\alpha$ a prehomogeneous dimension vector with generic representation $T=T_1^{\oplus\lambda_1}\oplus \dots \oplus T_m^{\oplus\lambda_m}$. Let $r=n-m$, and $S_{m+1},\dots,S_{n}$ the simple objects in $T^\perp$ with dimension vectors $\beta^{1},\dots,\beta^r$, respectively. For a fixed tuple $\underline{m}\in \nat^r$, denote by $b_\alpha$ the $b$-function of several variables $b_{\underline{m}}(\underline{s})$ of the semi-invariants $c_{S_{m+1}},\dots,c_{S_{n}}$. Let $c=-E^{-1}E^t$ be the Coxeter transformation, where $E$ is the Euler matrix of $Q$ (see \cite{elements}). 

We assume throughout that $c(\alpha)\in \nat^n$. We have the following formula by \cite[Theorem 5.3]{lol}

\begin{equation}
\mathlarger{b_{\alpha}=b_{c(\alpha)}\prod_{x\in Q_0} \dfrac{\displaystyle[s]^{\beta^1_x,\dots,\beta^r_x}_{\alpha_x}}{\displaystyle[s]^{\beta^1_x,\dots,\beta^r_x}_{c(\alpha)_x}}}.
\label{eq:refl}
\end{equation}

Put $S:=S_i$ for some $i\in\{m+1,\dots, n\}$ and $\beta$ its dimension vector. Following \cite{openorb}, we call $\alpha$ an $S$-\itshape stable \normalfont dimension vector if $T_i^{++}$ is a direct summand of $T$ (see (\ref{eq:defin})). By \cite{openorb}, in such case the zero set $Z(c^{S})$ is the closure of a codimension $1$ orbit. Let $\tau$ be the Auslander-Reiten translation (see \cite{elements}).

\begin{proposition}
Assume $\alpha$ is $S$-stable and $c(\alpha)$ is $\tau S$-stable (assume $S$ is not projective). Then $Z(c^{S})$ has rational singularities iff $Z(c^{\tau S})$ has rational singularities.
\end{proposition}

\begin{proof}
By (\ref{eq:refl}) and Theorem \ref{thm:ratio}, we see that it is enough to show that $\alpha \geq \beta$ and $c(\alpha) \geq \beta$. Since $\alpha$ is $S$-stable, we have the composite of injective maps $T_i \hookrightarrow T_i^{++} \hookrightarrow T$ and surjective maps $T_i \twoheadrightarrow \overline{T}_i\twoheadrightarrow S_i$, we get $\alpha \geq \beta$. Since $S\in ^\perp\!\!\!(\tau T)$ and $c(\alpha)$ is $\tau S$-stable we get dually that $c(\alpha)\geq \beta$.
\end{proof}

Note that if $S$ is projective, we reduce as in \cite[Proposition 5.4 b)]{lol} using reflection functors to the case when $S$ is simple, after which we obtain the root $-1$ of the $b$-function.

\begin{theorem}\label{thm:dynkin1}
If $Q$ is a Dynkin quiver, then all codimension $1$ orbit closures in $\Rep(Q,\alpha)$ have rational singularities.
\end{theorem}

\begin{proof}
By \cite{openorb}, for Dynkin quivers all dimension vectors are stable.
\end{proof}

\begin{remark}\label{rem:sing}
It is known that all orbit closures of Dynkin quiver of type $A$ and $D$ have rational singularities \cite{orb1,orb2,ryan}. For type $A$ quivers, the roots of the $b$-functions corresponding to codimension $1$ orbit closures are all integers (see the computations in \cite{sugi,lol}), while for type $D$ they can be half-integers \cite[Theorem 4.10]{lol}. We know that the global $b$-function is the same as the local $b$-function at $0$. (see \cite{lol}). By \cite[Theorem 1.4]{zwarahyper} orbit closures in these cases are hypersurfaces if and only if they are local hypersurfaces at $0$, and this is a property preserved by smooth morphisms. We conclude that the types of singularities of orbit closures of type $A$ and those of type $D$ are not equivalent.
\end{remark}

\begin{theorem}
\label{thm:extend1}
Let $Q$ be an extended Dynkin quiver, $T_1,\dots, T_m$ be pairwise non-isomorphic indecomposables in $\Rep(Q)$ such that $\Ext(T_i,T_j)=0$, for any $i,j\leq m$. Then there is a positive integer $N$ such that all codimension $1$ orbit closures in $\Rep(Q,\alpha)$ have rational singularities, for any dimension vector $\alpha = \lambda_1\cdot \underline{d}(T_1) + \dots + \lambda_m \cdot\underline{d}(T_m)$ with $\lambda_i\geq N$ for $i=1,\dots, m$.
\end{theorem}

\begin{proof}
By \cite[Proposition 5.7]{lol} we can compute the $b$-function using (\ref{eq:refl}) in a finite number of steps, and for large enough $N$ the dimension vectors are stable.
\end{proof}

Based on this, we make the following

\begin{conjecture}
Theorem \ref{thm:extend1} is true for any quiver $Q$.
\end{conjecture}

Now we illustrate how Proposition \ref{prop:link} can be used to determine the property of having rational singularities for nullcones. Assume the $b$-function of several variables $b(\underline{s})$ (we suppress the index $\underline{m}$) is of the form 
\begin{equation}
\mathlarger{b(\underline{s}) = \prod_{i=1}^r  [s]^{\underline{e}^i}_{d_i} \prod_{j=1}^N  [s]^{\gamma^j}_{a_j,b_j}},
\label{eq:ass}
\end{equation}
where $d_i\in \nat$ and for all $j=1,\dots,N$ we have $a_j,b_j \in \nat, \gamma^j \in \nat^r$ with $\underline{e}\cdot \gamma^j\leq a_j$.

\begin{remark}
For example, if $Q$ is Dynkin and the multiplicities $\lambda_i\geq N(Q)+1$ for $i=1,\dots,m$, then the $b$-function of several variables always looks as above (\ref{eq:ass}). The only thing left to see is that $\underline{e}\cdot \gamma^j\leq a_j$, for every $j=1,\dots, N$. This follows from the formula (\ref{eq:refl}), since the multiplicity condition implies $\mathcal{Z'}(Q,\alpha))\neq \emptyset$ by \cite{mater}, while this in turn implies that $\dim \alpha \geq \beta_1+\dots+\beta_r$ by \cite[Lemma 3.6]{zwara1}. 

Generalizing the case of $1$ semi-invariant, we conjecture that the condition $\underline{e}\cdot \gamma^j\leq a_j$ on $b$-functions of several variables of semi-invariants of more general prehomogeneous vector spaces implies rational singularities of their zero sets.
\end{remark}

\begin{definition}
We say an element $z\in Z(\tilde{B})$ is \textit{good}, if either $z=-\underline{e}$ or $\underline{e}\cdot z < -r$.
\end{definition}

\begin{proposition} Suppose $Z=Z(f_1,\dots,f_r)$ is a reduced complete intersection and all elements in $Z(\tilde{B})$ are good. Then $Z$ has rational singularities.
\end{proposition}

\begin{proof}
Localizing $\tilde{B}$ at $z=-\underline{e}$ and looking at the elements $b_{\underline{e}^{i}}$ for $i=1,\dots,r$, we see that the conditions $\underline{e}\cdot \gamma^j\leq a_j$ imply that $-r$ is a root of $\tilde{b}(s)$ with multiplicity $1$. Hence the largest root of the polynomial $\tilde{b}(s)/(s+r)$ is smaller than $-r$, so we conclude by combining Theorem \ref{thm:ratio} and Proposition \ref{prop:link}.
\end{proof}

Put $L:=\{\gamma^1,\dots, \gamma^N\}\cup \{\underline{e}^1,\dots,\underline{e}^r\}$. Set $\Gamma=L\backslash\{\underline{e}^{1},\dots,\underline{e}^{r}\}$, and for each $i\in\{1,\dots,r\}$, let $\Gamma_i:=\{j\in \{1,\dots,N\}: \gamma^j_i>0, \gamma^j\in\Gamma\}$.

\begin{lemma} \label{lem:reduc}
Using the notation above, we have the following reduction techniques:
\begin{itemize}
\item[(a)] Assume there exists $I=\{i_1,\dots, i_k\}\subset \{1,\dots,r\}$ with the property that for any $(j_1,\dots,j_k)\in \Gamma_{i_1}\times\dots \times \Gamma_{i_k}$ there exists $(u_1,\dots,u_k)\in \rat_+^k$ such that $\displaystyle\sum_{i\in \{j_1,\dots,j_k\}} u_i  \gamma^{j_i}=\underline{e}$ and $\displaystyle\sum_{i\in \{j_1,\dots,j_k\}} u_i (a_i+1)>r$. Assume further that for any $i\in I$, all $z\in Z(\tilde{B})$ with $z_i\leq -1$ are good. Then all the elements in $Z(\tilde{B})$ are good.
\item[(b)] Assume there exists $J=\{j_1,\dots, j_k\}\subset \{1,\dots, r\}$ such that $\underline{e}\notin E_J:=\op{span}_{\rat_+}\Gamma+\op{span}_{\rat}\{\underline{e}^{j_1},\dots,\underline{e}^{j_k}\}$, and suppose $J$ is maximal with this property. Write $J^c=J_+ \cup J_-$ with $J_+ = \{i \in J^c: \underline{e}\in E_J + \rat_+\cdot\underline{e}^i\}$ and $J_-=\{i \in J^c:\underline{e}\in E_J + \rat_-\cdot\underline{e}^i\}$.
Then for any $z\in Z(\tilde{B})$, there is an element $i\in J^c$ such that if $i\in J_+$ then $z_i\in \integ_{<0}$, and if $i\in J_-$ then $z_i\in \nat$.
\end{itemize}
\end{lemma}

\begin{proof}
\begin{itemize}
\item[(a)] Take an element $z\in Z(\tilde{B})$ such that $z_i>-1$, for all $i\in I$. We see that for $j=1,\dots, k$ the linear factors of $b_{\underline{e}^{i_j}}(\underline{s})$ involve the vectors from $\{\underline{e}^{i_j}\}\cup\Gamma_{i_j}$. Hence $z$ is a root of a linear factor of $b_{\underline{e}^{i_j}}(\underline{s})$ involving a vector from $\Gamma_{i_j}$, for each $j=1,\dots,k$. But then the condition $\displaystyle\sum_{i\in \{j_1,\dots,j_k\}} u_i (a_i+1)>r$ implies that $\underline{e}\cdot z < -r$, so $z$ is good.
\item[(b)] Take an element $z\in Z(\tilde{B})$. 
As in the proof of Proposition \ref{prop:link}, there exists $\underline{c}\in \integ^r$ such that $\underline{e}\cdot \underline{c}=1$, $\underline{c}_{j_i}=0$, for $i=1,\dots,k$, and $\max_{\gamma\in \Gamma}\{\gamma \cdot \underline{c}\}$ is arbitrary small. It is immediate that if $i\in J_+$, then $\underline{c}_i>0$, and if $i\in J_-$, then $\underline{c}_i<0$. Since $z$ is a root of $b_{\underline c}$, it follows that there is an element $i\in J^c$ such that $z$ is the root of a form involving a term $\underline{e}^i$, hence the conclusion.
\end{itemize}
\end{proof}

\begin{example}
\label{ex:first}
Take the following $E_6$ quiver and dimension vector:

\[\xymatrix{
 & & n+m\ar[d] & & \\
n \ar[r]& 2n+m\ar[r] & 2n+m & \ar[l] 2n+m& \ar[l] n
}
\]

The generic representation is $T=T_1^{\oplus m}\oplus T_2^{\oplus n}$, where $T_1=\arraycolsep=1.5pt\begin{array}{ccccc}
&&1&& \\
0 & 1 & 1 & 1 & 0
\end{array}, T_2=\begin{array}{ccccc}
&&1&& \\
1 & 2 & 2 & 2 & 1
\end{array}$. The simples in the right perpendicular category $T^\perp$ are the indecomposables with dimension vectors:

\[\arraycolsep=1.5pt
\begin{array}{ccccc}
&&1&& \\
1 & 1 & 1 & 0 & 0
\end{array}\,,\,
\begin{array}{ccccc}
&&1&& \\
0 & 0 & 1 & 1 & 1
\end{array}\,,\,
\begin{array}{ccccc}
&&0&& \\
0 & 1 & 1 & 1 & 1
\end{array}\,,\,
\begin{array}{ccccc}
&&0&& \\
1 & 1 & 1 & 1 & 0
\end{array}.
\]

By Theorem \ref{thm:dynk}, if $n,m\geq 2$ the nullcone is reduced and a complete intersection. Using (\ref{eq:refl}) repeatedly we obtain the $b$-function of several variables:

$$b(\underline{s})= [s]^{1000}_{n+m} \cdot [s]^{0100}_{n+m}\cdot[s]^{0010}_n\cdot[s]^{0001}_n \cdot [s]^{0011}_{n,2n+m} \cdot [s]^{0110}_{n+m,2n+m}\cdot [s]^{1001}_{n+m,2n+m}.$$

We want to show that for $n,m\geq 2$, each element in $z\in Z(\tilde{B})$ is good. We see that Lemma \ref{lem:reduc} (a) applies with $I=\{1,2\}$ and we have $(n+m+1)+(n+m+1)>4$. Hence we can assume one of the 2 possibilities: $z_1\leq -1$ or $z_2\leq -1$. We consider the first case (the latter is analogous), put $z_1=-k_1$, with $k_1\geq 1$, (moreover, we can assume $k_1\leq d_1=n+m$). Then put $z'=(z_2,z_3,z_4)$ and evaluating $s_1=-k_1$ we consider the $b$-function of several variables

$$b_1= [s]^{100}_{n+m}\cdot[s]^{010}_n\cdot[s]^{001}_n \cdot [s]^{011}_{n,2n+m} \cdot [s]^{110}_{n+m,2n+m} \cdot [s]^{001}_{n+m-k_1,2n+m-k_1}.$$

If $\tilde{B}_1$ is the ideal associated to $b_1$, then $z'\in Z(\tilde{B}_1)$. Now we can apply Lemma \ref{lem:reduc} (b) with $J=\emptyset$. Then $J_+=\{1,3\}$ and $J_-=\{2\}$. Hence we have another 3 possibilities. Assume first that $z'_3=-k_2$, with $k_2\geq 1$ (the case of $1\in J_+$ is analogous). This leads to

$$b_2= [s]^{10}_{n+m}\cdot[s]^{01}_n \cdot [s]^{01}_{n-k_2,2n+m-k_2} \cdot [s]^{11}_{n+m,2n+m}.$$

Let $z''=(z'_1,z'_2)$. Again, Lemma \ref{lem:reduc} (a) can be applied with $I=\{1\}$, since we have $n+m+1>2$. Hence we can assume $z''_1=-k_3$, $1\leq k_3\leq n+m$. Then we are left with  $b_3 = [s]^{1}_{n}\cdot [s]^{1}_{n-k_2,2n+m-k_2} \cdot [s]^1_{n+m-k_3,2n+m-k_3}$. Hence for the last choice $k_4\geq \min\{1,n-k_2+1\}$. Hence $\underline{e}\cdot z = -k_1-k_2-k_3-k_4 \leq -4$, with equality only for $z=-\underline{e}$, hence $z$ is good.

Now we return to $b_1$ and we are left with the case $2\in J_-$, so put $z'_2=k_2$, with $k_2\geq 0$. Due to the previous discussion we can assume that $z'_1>-1$ and $z'_3>-1$. Hence we consider

$$b'_2=[s]^{10}_{n+m}\cdot[s]^{01}_n \cdot [s]^{01}_{n+k_2,2n+m+k_2} \cdot [s]^{10}_{n+m+k_2,2n+m+k_2} \cdot [s]^{01}_{n+m-k_1,2n+m-k_1}.$$

Put $z''=(z'_1,z'_3)$. Since $z''_1>-1$, we conclude that $z''$ is not a root of $(b'_2)_{\underline{e}^1}$, finishing this last case.

In conclusion, if $n,m\geq 2$ the nullcone is a reduced complete intersection with rational singularities. Note that if $n,m\geq 3$, then it is the closure of an orbit.
\end{example}

Next, we state a result for tree quivers, that is, quivers whose underlying graphs contain no cycles.

\begin{theorem}
\label{thm:tree} Let $Q$ be a tree quiver, $\alpha$ a dimension vector, and $c_{W_1},c_{W_2}$ two irreducible semi-invariants on $\Rep(Q,\alpha)$ with $\underline{d}(W_i)_x\leq 1$, for all $x\in Q_0$. Then the hypersurface $Z(c_{W_i})$ has rational singularities. Moreover, if either $\alpha$ or $\underline{d}(W_1)+\underline{d}(W_2)$ is prehomogeneous, then for large enough $N$ the zero-set $Z(c_{W_1},c_{W_2})\subset \Rep(Q,N\cdot\alpha)$ has rational singularities.
\end{theorem}

\begin{proof}
The first part follows from computation of the $b$-function using the slice technique \cite[Theorem 4.8]{lol} and the fact that $c_{W_i}$ are irreducible.

For the second part, the assumptions imply multiplicity-freeness (see \cite{lol}) hence we can consider the $b$-function of $2$ variables $b_{\underline{m}}(\underline{s})$. Specializing $b_{(1,0)}(s_1,0)$ and $b_{(0,1)}(0,s_2)$, we see again by \cite[Theorem 4.8]{lol} that the set $L$ of all linear forms appearing in $b_{\underline{m}}(\underline{s})$ are from $L\subset\{(1,0),(0,1),(1,1)\}$ and that the constants $a_i,b_i$ from (\ref{eq:ass}) increase linearly in $N$. Hence we can apply Lemma \ref{lem:reduc} (a) with $I=\{1\}$, say, and for large enough $N$ we obtain that all elements in $Z(\tilde{B})$ are good.
\end{proof}

\begin{example}
\label{ex:pos}
Here we give an example where Theorem \ref{thm:tree} is not applicable and show that $\tilde{b}(s)$ can have positive roots already for $2$ semi-invariants. We continue Example \ref{ex:notred} with the dimension vector

\[\xymatrix{
& & 3n \ar[d] & & & &\\
2n \ar[r] & 4n \ar[r] & 7n & 4n \ar[l] & 3n \ar[l] & 2n \ar[l] &n \ar[l]
}\]

Let $S_1$ and $S_2$ be the indecomposables $\arraycolsep=1.5pt
\begin{array}{ccccccc}
 &  & 0 & & & & \\
0&0 & 1 &1&1&1&1 
\end{array}\,,\,
\begin{array}{ccccccc}
 &  & 1 & & & & \\
0&1 & 2 &1&1&1&0 
\end{array}$, respectively. The $b$-function of $2$ variables of  $c_{S_1},c_{S_2}$ is

$$b(\underline{s})=[s]^{01}_{4n} \cdot ([s]^{01}_{n,3n})^2 \cdot [s]^{01}_{2n,4n} \cdot [s]^{10}_n \cdot [s]^{11}_{n,4n}\cdot[s]^{12}_{4n,7n}$$

It is easy to see that for any $\underline{c}\in \integ^2$ with $c_1+c_2=1$, we have that either $s_1+2s_2+4n+1$ or $s_1+s_2-2n$ is a factor of $b_{\underline{c}}(\underline{s})$. Hence $(8n+1,-6n-1)\in Z(\tilde{B})$ and $2n$ is a root of $\tilde{b}(s)$.
\end{example}

\begin{remark}
We point out that most of the results in this section are valid over an arbitrary field $k$ of characteristic $0$. Consider the zero set $Z=Z(f_1,\dots,f_r)$ of some fundamental semi-invariants of a quiver $Q$ with a prehomogeneous dimension vector $\alpha$. Assume for simplicity that $Z$ is reduced, irreducible with rational singularities over $\complex$ (hence by Remark \ref{characteristic} $Z$ is irreducible, reduced over $k$). We claim that $Z$ has rational singularities also over $k$. One way we can see this as follows. Since fundamental semi-invariants have linearly independent weights, there exists a reductive group $G$, with $\SL(\alpha)\subset G\subset \GL(\alpha)$, such that $Z$ is the nullcone for the action of $G$ on $\Rep(Q,\alpha)$. By \cite{hesselink}, there exists a desingularization $Z'\to Z$, where $Z'$ is the total space of a vector bundle on a flag variety. We note that $Z'\to Z$ is defined already over $\rat$. Moreover, in this situation the property of rational singularities is equivalent to the vanishing of the cohomology of certain vector bundles on this flag variety (see \cite{jerzy}). This is independent of the field, as long as it is of characteristic $0$.
\end{remark}

\section*{Acknowledgement}
I would like to express my gratitude towards my advisor, Prof. Jerzy Weyman, for providing me many valuable insights.

\vspace{0.05in}

\bibliographystyle{amsplain}
\bibliography{biblo}

\end{document}